\documentclass[a4paper,11pt]{article}
\usepackage{amsmath}
\usepackage{amsthm}
\usepackage{amsfonts}
\usepackage{amssymb}
\usepackage[left=2.5cm,top=2.4cm,right=2.5cm,bottom=2.4cm]{geometry}

\newtheorem{lem}{Lemma}[section]
\newtheorem{thm}{Theorem}[section]

\numberwithin{equation}{section}

\begin{document}
\title{Convergence rates for a  supercritical branching process in a random environment }
\author{Chunmao HUANG$^{a,b}$, Quansheng LIU$^{a,c,}$\footnote{Corresponding author at:  LMBA, UMR 6205, Universit\'e de Bretagne-Sud, Campus de Tohannic,
BP 573, 56017 Vannes, France.  \newline \indent \ \ Email addresses: sasamao02@gmail.com (C. Huang), quansheng.liu@univ-ubs.fr (Q. Liu).}
\\
\small{\emph{$^{a}$LMBA, Universit\'e de Bretagne-Sud, Campus de Tohannic,
BP 573, 56017 Vannes, France}}\\
\small{\emph{ $^{b}$ CMAP, Ecole Polytechnique,  Route de Saclay
                     91128 Palaiseau, France}}\\
\small{\emph{ $^{c}$
 Changsha University of Science and Technology}}, \\
            \small{\emph{ School of Mathematics and Computing Science, Changsha, 410076, P.R. China}}}
\date{February 16, 2013}
\maketitle

\begin{abstract}
Let $(Z_n)$ be a supercritical branching process in a stationary and ergodic random environment
$\xi$.
We study the convergence rates of the martingale $W_n = Z_n/ \mathbb{E}[Z_n| \xi]$
to its limit
 $W$. The following results about the convergence almost sure (a.s.), in law or in
probability, are shown.    (1) Under a moment condition of order $p\in (1,2)$,
$W-W_n = o (e^{-na})$ a.s. for some
$a>0$ that we find explicitly; assuming only   $\mathbb{E} W_1 \log  W_1^{\alpha
+1} < \infty$
for some $\alpha >0$, we have $W-W_n = o (n^{-\alpha})$ a.s.; similar conclusions
hold for a branching process in a varying environment.   (2) Under a second moment
condition, there are norming constants $a_n(\xi)$ (that we calculate explicitly)
such that $a_n(\xi) (W-W_n)$ converges in law to a non-degenerate distribution.
(3) For a branching process in a finite state random environment,  if $W_1$ has a
finite exponential moment, then so does $W$,  and the decay rate of
$\mathbb{P}(|W-W_n| > \epsilon)$ is supergeometric.
\\*



\end{abstract}

\section{Introduction and main results}
The study  of branching processes is interesting  due to a wide
range of applications: see for example the books by Harris (1963, \cite{Ha}) and
Athreya \& Ney (1972, \cite{a}). In a Galton-Watson process,
particles behave independently, each gives birth to a random number of
particles of the next generation with a fixed distribution.
 A branching process in a random environment is a natural and
important extension of the Galton -Watson process, where the offspring distributions vary according to a random environment indexed by the time. This model was first introduced by  Smith \& Wilkinson (1969, \cite{smith}) in the independent and identically distributed environment case, and  by Athreya \& Karlin (1971, \cite{a1}) in the  stationary and ergodic environment case. Since then many authors have contributed to the subject. For recent results, see for example
Afanasyev, Geiger, Kersting \& Vatutin (2005, \cite{af1}),  Kozlov (2006, \cite{kozlov}), Bansaye \& Berestycki (2009, \cite{ba}),
 and B\"oinghoff,  Dyakonova, Kersting \& Vatutin (2010, \cite{bd}),  among others.  Here,  for a supercritical branching process $(Z_n)$  in a stationary and ergodic random environment $\xi$,
   we study the convergence rates of the  martingale $W_n = Z_n/ \mathbb{E}[Z_n| \xi]$
to its limit $W$, by considering the almost sure (a.s.) convergence (find $a>0$ such that $W-W_n = o (e^{-na})$   or $W-W_n = o (n^{-a})$ a.s.),  the convergence in law (find norming constants $a_n(\xi)$
such that $a_n(\xi) (W-W_n)$ converges in law to a non-degenerate distribution),
and the convergence in probability (give an estimation of the deviation probability $\mathbb{P}(|W-W_n| > \epsilon)$). These results extend the corresponding ones of Asmussen (1976, \cite{soren}), Heyde (1970, \cite{heyde70}), and Athreya (1994, \cite{ath}) on the Galton-Watson process.\\*

A branching process in a stationary and ergodic random environment can be described as follows. Let $ \xi=(\xi_0,\xi_1,\xi_2,\cdots)$ be a stationary and ergodic
sequence. Suppose that each realization of $\xi_n$ corresponds to a probability distribution on $\mathbb{N}_0=\{0,1,2,\cdots\}$ denoted by
$p(\xi_n)=\left\{p_i(\xi_n):  i\in\mathbb{N}_0\right\}$, where
$$
 \text{$p_i(\xi_n)\geq0$, \quad$\sum_ip_i(\xi_n)=1$\quad and \quad$\sum_iip_i(\xi_n)\in(0,\infty)$. }
$$
A branching process $(Z_n)$ in the random environment $\xi$ is a class of branching processes in a varying environment indexed by $\xi$. By definition,
$$Z_0=1,\qquad Z_{n+1}=\sum_{i=1}^{Z_n}X_{n,i}\quad (n\geq0),$$
where given the environment $\xi$, $X_{n,i}$ ($n\geq0, i\geq1$) is a sequence of (conditionally) independent random variables;
each $X_{n,i}$ has distribution
$p(\xi_n)$.

Given $\xi$, the conditional probability will be denoted by
$\mathbb{P}_{\xi}$ and the corresponding expectation by $\mathbb{E}_{\xi}$. The total
probability will be denoted by $\mathbb{P}$ and the corresponding expectation
by $\mathbb{E}$. As usual, $\mathbb{P}_\xi$ is called \emph{quenched law}, and $\mathbb{P}$ \emph{annealed law}.

Let $\mathcal {F}_0=\mathcal
{F}(\xi)=\sigma(\xi_0,\xi_1,\xi_2,\cdots)$ and $\mathcal
{F}_n=\mathcal
{F}_n(\xi)=\sigma(\xi_0,\xi_1,\xi_2,\cdots,X_{k,i},\;0\leq k\leq
n-1,\;i=1,2,\cdots)$ be the $\sigma$-field generated by the random
variables $X_{k,i},\;0\leq k\leq n-1,\;i=1,2,\cdots$, so that
 $Z_n$ are $\mathcal {F}_n$-measurable.
For $n\in \mathbb N_0$ and $p\geq1$, set
$$m_n(p)= \sum_{i=1}^\infty i^pp_i(\xi_n)\qquad\text{and}\qquad
m_n=m_n(1).$$
Then $m_n(p)$ is the $p$-th moment of the offspring distribution in generation $n$ and $m_n$ is the mean offspring number of an individual in generation $n$, given the environment $\xi$.
Let
$$\Pi_0=1\qquad \text{and}\qquad
\Pi_n=\prod_{i=0}^{n-1}m_i \quad\text{for $n\geq1$}.$$
Then $\Pi_n=\mathbb{E}_{\xi}Z_n$. It is well known that the normalized population size
$$W_n=\frac{Z_n}{\Pi_n}$$
is a nonnegative martingale under $\mathbb{P}_\xi$  for every $\xi$ with
respect to the filtration $\mathcal{F}_n$, and
$$\lim_{n\rightarrow\infty}W_n=W$$
exists a.s. with $\mathbb{E}W\leq1$. In the present paper, we consider the \emph{supercritical} case where
$\mathbb{E}\log m_{0}\in (0,\infty )$,   and we are interested in the convergence rates of $W-W_n$.
\\*

We first consider the a.s. convergence rate.
For a Galton-Watson process,  Asmussen (1976,
\cite{soren}) showed that $W-W_n=o(m^{-n/q})$ a.s. and $\mathbb{P}(W>0)>0$, if and only if
$\mathbb{E}Z_1^p<\infty$ , where $1<p<2$, $1/p+1/q=1$ and $m=\mathbb{E}Z_1\in(1,\infty)$,
and that
$W-W_n=o(n^{-\alpha})$ a.s. if
$\mathbb{E}Z_1(\log^+Z_1)^{1+\alpha}<\infty$ for some $\alpha>0$.  The following two theorems show that similar results hold for a branching process in a random environment.

\begin{thm}\label{CRT2.1.1}
If $\mathbb{E}\log
\mathbb{E}_{\xi}(\frac{Z_1}{m_0})^{p}<\infty$ for some
$p\in(1,2)$, then for any $\varepsilon>0$,
\begin{equation*}
W-W_n=o(m^{- \frac{n}{q+\varepsilon}})\qquad a.s.,
\end{equation*}
where $m=\exp(\mathbb{E}\log m_0)>1$ and $1/p+1/q=1$.
\end{thm}

\noindent \textbf{Remark.} As $\mathbb{E}\log m_0\in(0,\infty)$, it can be seen that   $\mathbb{E}\log
\mathbb{E}_{\xi}(\frac{Z_1}{m_0})^{p}<\infty$ if
$\mathbb{E}\log^+\mathbb{E}_{\xi}Z_1^{p}<\infty$.
\\*

\begin{thm}\label{CRT2.1.2} 
Assume that
$\mathbb{E}\frac{Z_1}{m_0}(\log^+\frac{Z_1}{m_0})^{1+\alpha}<\infty$ for some $\alpha>0$. Then
\begin{equation*}
W-W_n=o(n^{-\alpha})\qquad a.s.;
\end{equation*}
moreover,  the series $\sum_n (W -W_n)$ converges $a.s.$ if $\alpha\geq1$.
\end{thm}

We shall see after the proof of Theorem \ref{CRT2.1.2} that the condition
$\mathbb{E}\frac{Z_1}{m_0}(\log^+\frac{Z_1}{m_0})^{1+\alpha}<\infty$
 can be replaced by $\mathbb{E}\frac{Z_1}{m_0}$$(\log^+ {Z_1} )^{1+\alpha}$$<\infty$.\\

We next show that under a second moment condition, with an appropriate normalization, $W-W_n$ converges in law to a non-trivial distribution. Recall that for a Galton-Watson process, Heyde (1970, \cite{heyde70}) proved that if $m=\mathbb{E}Z_1\in(1,\infty)$ and $\sigma^2=\mbox{Var}{Z_1}\in(0,\infty)$, then $m^{n/2}(W-W_n)$ converges in law to a non-degenerate distribution. We shall prove that  similar results hold for a branching process in a random environment.

 Let
\begin{equation}
\delta^2_{\infty}(\xi)=\sum_{n=0}^\infty\frac{1}{\Pi_n}\left(\frac{m_n(2)}{m_n^2}-1\right).
\end{equation}
It can be easily checked that $\delta^2_\infty(\xi)>0$ a.s. if and only if $p_i(\xi_0)<1$ a.s. for all $i\in \mathbb N_0$, which means that  the  offspring distributions are non-degenerate. By a result in Fearn (1972, \cite{fe}) or Jagers (1974, \cite{jajer}),
 $\delta^2_{\infty}(\xi)$ is the variance of $W$ under $\mathbb{P_\xi}$ if the series converges.
  Indeed, by the orthogonality of martingales, it is not hard to see that
$$\mathbb{E}_\xi W_n^2=\sum_{k=0}^{n-1}\mathbb{E}_\xi(W_{k+1}-W_k)^2+1=\sum_{k=0}^{n-1}\frac{1}{\Pi_k}\left(\frac{m_k(2)}{m_k^2}-1\right)+1.$$
Therefore the martingale $\{W_n\}$ is bounded in $L^2$ under $\mathbb{P}_\xi$  if and only if the series converges; when it converges, $\delta^2_{\infty}(\xi)$ is the variance of $W$ under $\mathbb{P}_\xi$.
In Lemma \ref{CRL2.2.1} we shall see that  the series $\sum_{n=0}^\infty\frac{1}{\Pi_n}\left(\frac{m_n(2)}{m_n^2}-1\right)$ converges a.s. if $\mathbb{E}\log^+\mathbb{E}_\xi\left(\frac{Z_1}{m_0}-1\right)^2<\infty$, or equivalently, $\mathbb{E}\log\mathbb{E}_\xi\left(\frac{Z_1}{m_0}\right)^2<\infty$.

As usual, we write $T^n\xi=(\xi_n, \xi_{n+1}, \cdots)$ if $\xi=(\xi_0, \xi_1, \cdots)$ and $n\in \mathbb N_0$. We have the following theorem.

\begin{thm}\label{CRASTP}
Assume that  $p_i(\xi_0)<1$ a.s. for all $i\in\mathbb{N}_0$. For $x\in \mathbb R$,
let $\Phi_1(x)=\mathbb{P}_\xi(G\sqrt{W}\leq x)$ and $\Phi_2(x)=\mathbb{P}(G\sqrt{W}\leq x)=\mathbb{E}\Phi_1(x)$, where $G$ is a gaussian random variable with distribution $\mathcal{N}(0,1)$, independent of $W$ under $\mathbb{P}_\xi$. If $\mathbb{E}\log\mathbb{E}_\xi\left(\frac{Z_1}{m_0}\right)^2<\infty$, then
\begin{equation}\label{CRASEP1}
\sup_{x\in\mathbb{R}}\left|\mathbb{P}_\xi\left(\frac{\sqrt{\Pi_n}(W-W_n)}{\delta_\infty(T^n\xi)}\leq x\right)-\Phi_1(x)\right|\rightarrow0\quad\text{in $L^1,$}
\end{equation}
and
\begin{equation}\label{CRASEP2}
\sup_{x\in\mathbb{R}}\left|\mathbb{P}\left(\frac{\sqrt{\Pi_n}(W-W_n)}{\delta_\infty(T^n\xi)}\leq x\right)-\Phi_2(x)\right|\rightarrow0.
\end{equation}
\end{thm}

\noindent\textbf{Remark.} Let $\Phi(x)=\frac{1}{\sqrt{2\pi}}\int_{-\infty}^xe^{-t^2/2}dt$ be the distribution function of standard normal distribution $\mathcal{N}(0,1)$. It can be easily seen that
$$\Phi_1(x)=\mathbb{E}_\xi\Phi({x}/{\sqrt{W}})\mathbf{1}_{\{W>0\}}+\mathbb{P}_\xi(W=0)\mathbf{1}_{\{x\geq0\}}.$$

\bigskip

 In fact, (\ref{CRASEP1}) is a quenched version of a convergence result in law: it states that the quenched law of $U_n=\frac{\sqrt{\Pi_n}(W-W_n)}{\delta_\infty(T^n\xi)}$ converges in some sense to a non-degenerate distribution; (\ref{CRASEP2}) is a annealed version of convergence in law: it says that the annealed law of $U_n$ converges to a non-degenerate distribution.

 For the Galton-Watson process, (\ref{CRASEP1}) and (\ref{CRASEP2}) reduce to the result of Heyde (1970, \cite{heyde70}).

 We mention that if we change the norming, then we can obtain central limit theorems. In fact, in extending the theorems of Hedye (1971, \cite{heyde1}) and Heyde and Brown (1971, \cite{heyde2}) on the Galton-Watson process,   Wang, Gao \& Liu (2011, \cite{gao}) have recently shown that (\ref{CRASEP1}) and (\ref{CRASEP2}) hold with $\Pi_n$ replaced by $Z_n$ and $\Phi_i$ ($i=1,2$) replaced by $\Phi(x)=\frac{1}{\sqrt{2\pi}}\int_{-\infty}^xe^{-t^2/2}dt$. Compared to this result, the advantage of Theorem \ref{CRASTP} is that the norming therein  depends only on the environment.
\\*

We finally  consider a special case and show a super-geometric convergence rate. We consider a finite state random environment, namely,  each
$\xi_n$ takes values in a finite set $\{a_1,a_2,\cdots,a_N\}$. Moreover, we assume that $p_0(\xi_0)=0$ and $p_1(\xi_0)<1$ a.s. (which exclude extinction). For this special model, we have a super-geometric convergence rate in probability under an exponential moment condition.

\begin{thm}\label{CRT2.1.9} 
For a branching process in a finite state random environment with $p_0(\xi_0)=0$ and $p_1(\xi_0)<1$ a.s., we have
\begin{itemize}
\item[] (a) $\mathbb{E}e^{\theta W}<\infty$ for some $\theta>0$ if and only if $\mathbb{E}e^{\theta_0W_1}<\infty$ for some $\theta_0>0$ ;
\item[] (b) if $\mathbb{E}e^{\theta_0W_1}<\infty$ for
some constant $\theta_0>0$, then there exist constants $\gamma>0$
and $C>0$ such that
\begin{equation}\label{CRE2.1.3}
\mathbb{P}_\xi (|W-W_n|>\varepsilon)\leq C\exp({-\gamma
\Pi_n^{1/3}\varepsilon^{2/3}})\qquad a.s.,\end{equation}
and
\begin{equation}\label{CRE2.1.4}
\mathbb{P} (|W-W_n|>\varepsilon)\leq C\exp({-\gamma
\underline{m}^{n/3}\varepsilon^{2/3}}),
\end{equation} where
$\underline{m}:=\emph{ess\,inf} \ {m_0}>1$.
\end{itemize}
\end{thm}

For the classical Galton-Watson process, (\ref{CRE2.1.3}) reduces to the result of Athreya (1994, \cite{ath}, Theorem 5). Notice that by Fatou's Lemma and Jensen's inequality, we have
\begin{equation}\label{moment}
\mathbb{E}e^{\theta W}=\sup_n\mathbb{E}e^{\theta W_n}.
\end{equation}
So Theorem \ref{CRT2.1.9}(a) is an extension of a result of Athreya (1994, \cite{ath}, Theorem 4) and Liu (1996, \cite{liu96}, Theorem 2.1) on the classical  Galton-Watson process.
\\*

The rest of this paper is organized as follows. In Section \ref{CR1S1}, we consider a branching process in a varying environment and  show the a.s. convergence rate of $W-W_n$. In Sections \ref{CR1S2}-\ref{CR1S4}, we consider a branching process in a random environment and prove the main results:  we prove Theorems \ref{CRT2.1.1} and \ref{CRT2.1.2}   in Section \ref{CR1S2}, Theorem \ref{CRASTP} in Section \ref{CR1S3} and Theorem \ref{CRT2.1.9} in Section  \ref{CR1S4}.

\section{Branching process in a varying environment}\label{CR1S1}
In this section, we consider a branching process $(Z_n)_{n\geq0}$ in a varying
environment.  By definition,
$$Z_0=1,\qquad Z_{n+1}=\sum_{i=1}^{Z_n}X_{n,i}\quad (n\geq0),$$
where $(X_{n,i})_{i\geq1}$ are independent on some probability space $(\Omega, \mathbb{P})$; each $X_{n,i}$ has distribution
$p(n)=\{p_i(n):i\in\mathbb{N}_0\}$, where
$$
 \text{$p_i(n)\geq0$, \quad$\sum_ip_i(n)=1$ \quad and \quad$\sum_iip_i(n)\in(0,\infty)$.}
 $$

 Let $\mathcal {F}_0=\{\emptyset,\Omega\}$ and $\mathcal
{F}_n=\sigma(X_{k,i}:\;0\leq k\leq n-1,\;i=1,2,\cdots)$,  so that $Z_n$ are
$\mathcal {F}_n$-measurable.
For $n\in \mathbb N_0$ and $p\geq1$, set
\begin{equation}
m_n(p)=\mathbb{E}X_{n,i}^p=\sum_{i=1}^\infty i^p p_i(n),\qquad m_n=m_n(1),
\end{equation}
and
\begin{equation}
\Pi_0=1,\qquad \Pi_n=\mathbb{E}Z_n=\prod_{i=0}^{n-1}m_i\;\;(n\geq 1).
\end{equation}
Then the normalized population size
$$W_n=\frac{Z_n}{\Pi_n}$$
is a nonnegative martingale with
respect to the filtration $\mathcal{F}_n$, and the limit
$$W=\lim_{n\rightarrow\infty}W_n\qquad a.s.$$ exists with $\mathbb{E}W\leq1$. It is  known that there is a
nonnegative but possibly infinite random variable $Z_\infty$ such
that $Z_n\rightarrow Z_\infty$ in distribution as
$n\rightarrow\infty$. We are interested in the supercritical case
where $\mathbb{P}(Z_\infty=0)<1$, so that by (\cite{jajer},  Corollary 3), either
$\sum_{n=0}^\infty(1-p_1(n))<\infty$, or
$\lim_{n\rightarrow\infty}\Pi_n=\infty$. Here we consider the usual case where
$\lim_{n\rightarrow\infty}\Pi_n=\infty$.
\\*

Let $\bar X_n=X_n/m_n$, where $X_n$ has distribution $\{p_i(n):i\in\mathbb{N}_0\}$ (the offspring distribution of particles of $n$-th generation).\\

We first give a sufficient condition for an  exponential convergence rate of $W-W_n$.

\begin{thm}\label{CRT1.1}
Let $p\in(1,2)$ and  $\frac{1}{p}+\frac{1}{q}=1$.
\begin{itemize}
 \item[](i) If $\Pi_n\uparrow\infty$ and
$\sum_n\frac{1}{\Pi_n^{
\varepsilon/p}}\frac{m_n(p+\varepsilon)}{m_n^{p+\varepsilon}}<\infty$
for some $0<\varepsilon<2-p$, then
\begin{equation}\label{CRE1.1.2}
W-W_n=o(\Pi_n^{-{1}/{q}})\qquad a.s..
\end{equation}
\item[](ii) If $ a:=\liminf_{n}\frac{\log
\Pi_n}{n}\in(0,\infty)$ and $\sum_n
m^{-\frac{\varepsilon'n}{p}}\frac{m_n(p+\varepsilon)}{m_n^{p+\varepsilon}}<\infty$
for some $0< \varepsilon'<\varepsilon<2-p$, where $m:=e^a$,
then
\begin{equation}\label{CRE1.1.3}
W-W_n=o(m^{-{n}/{q}})\qquad a.s..
\end{equation}
\end{itemize}
\end{thm}

We next  give a sufficient condition for a polynomial convergence rate of $W-W_n$.

\begin{thm}\label{CRT1.2}
Let $\alpha>0$. Assume that $ \liminf_{n}\frac{\log
\Pi_n}{n}=a>0$ and
\begin{equation}
\label{log_moment}
\sum_{n=1}^\infty \frac{ \mathbb{E}\bar X_n(\log^+\bar X_n)^{1+\alpha+\varepsilon }   }{n^{1+\varepsilon}}<\infty
\end{equation}
for some $\varepsilon>0$.  Then
\begin{equation}\label{CRE1.1.4}
 W-W_n=o(n^{-\alpha})\qquad a.s.;
\end{equation}
moreover, when $\alpha\geq1$, the series $\sum_n (W -W_n)$ converges $a.s..$
\end{thm}

\noindent\textbf{Remark.}
 In Theorem \ref{CRT1.2},
 the moment condition ($\ref{log_moment}$)  can be replaced by
\begin{equation}
\label{log_moment2}
\sum_n \frac{   \mathbb{E}\bar X_n(\log^+\bar X_n)^{1+\alpha} (\log^+ \log^+ \bar X_n)^{1+\varepsilon}  }{n(\log
n)^{1+\varepsilon}}<\infty,
\end{equation}
or
\begin{equation}
\label{log_moment3}
\sum_n \frac{\mathbb{E}\bar X_n (\log^+\bar X_n)^{1+\alpha} (\log^+\log^+ \bar X_n) (\log^+ \log^+\log^+ \bar X_n)^{1+\varepsilon}  }{n (\log n) (\log \log
n)^{1+\varepsilon}}<\infty,
\end{equation}
etc, for some $\varepsilon >0$, where $n$ is sufficiently large.
This will be easily  seen in the proof.

\medskip
In fact, the above two theorems are easy consequences of the following more general and interesting result, which gives a sufficient condition for the convergence rate of $W-W_n$ to be $o(a_n^{-1})$, for any
given sequence $(a_n)$ increasing to $+\infty$ with $n$.

\begin{thm}\label{CRP2.3}
Let $(a_n)$ be a positive sequence of real numbers satisfying $a_n\uparrow\infty$. If there is an increasing function $g: [0,\infty) \rightarrow (0,\infty)$ such that
the function $x\mapsto \frac{x}{g(x)}$ is also increasing on $[0,+\infty)$ and that
\begin{equation}\label{CRPE}
\sum_{n=0}^{\infty} \frac{a_n}{g\left(\frac{\Pi_n}{a_n}\right)}\mathbb E \bar X_n g(\bar X_n)<\infty,
\end{equation}
then
$$W-W_n=o(a_n^{-1})\qquad a.s..$$
\end{thm}

Before giving the proof of Theorem \ref{CRP2.3}, let us first show how this theorem
implies Theorems \ref{CRT1.1} and \ref{CRT1.2}.

\begin{proof}[Proof of Theorem \ref{CRT1.1}]
(i) Let $a_n=\Pi_n^{1/q}$, and define
  $g(x)=x^{p+\varepsilon-1}$ for $x\geq  1$,  and $g(x)=1$ for $x\in [0,1]$. Clearly, $x\mapsto g(x)$ and $x\mapsto \frac{x}{g(x)}$ are increasing on $[0,\infty)$.
   Notice that $\frac{\Pi_n}{a_n}= {\Pi_n}^{1/p} \uparrow \infty$, and that
   \begin{eqnarray*}
\frac{a_n} {g\left(\frac{\Pi_n}{a_n}\right)} = \frac{ \Pi_n^{1/q} } {{\Pi_n}^{( p+\varepsilon-1)/p}  }
  = \frac{1}{  \Pi_n^{\varepsilon/p}}.
  \end{eqnarray*}
Notice also that, since $ \mathbb E \bar X_n = 1$, we have  $\mathbb E {\bar X_n}^{p+\varepsilon} \geq 1$, so that
  $$ \mathbb E \bar X_n g(\bar X_n) \leq 1 + \mathbb E {\bar X_n}^{p+\varepsilon} \leq 2  \mathbb E {\bar X_n}^{p+\varepsilon} = 2 \frac{m_n(p+\varepsilon)}{m_n^{p+\varepsilon}}.$$
Therefore Part (i) of   Theorem \ref{CRT1.1} is a direct consequence of Theorem \ref{CRP2.3}.

(ii) Let $a_n=m^{n/q}$ and define $g$ as in Part (i) above.
Since $\liminf_n\frac{1}{n}\log \Pi_n=a$, we have for all $\delta>0$ and  $n$ large enough,
$$\Pi_n>e^{(a-\delta)n}.$$
 Take $\delta>0$ small enough such that
$$\frac{p+\varepsilon}{q}a-(a-\delta)(p+\varepsilon-1)<-\frac{\varepsilon'}{p}a.$$
Then for $n$ large enough,
\begin{eqnarray*}
\frac{a_n}{g\left(\frac{\Pi_n}{a_n}\right)}=m^{\frac{p+\varepsilon}{q}n}\Pi_n^{-(p+\varepsilon-1)}
\leq \exp{\left(\left[\frac{p+\varepsilon}{q}a-(a-\delta)(p+\varepsilon-1)\right]n\right)}<\exp{\left(-\frac{\varepsilon'}{p}an\right)}.
\end{eqnarray*}
Therefore,  the convergence of the series  $\sum_n \frac{a_n}{g\left(\frac{\Pi_n}{a_n}\right)}\mathbb E \bar X_n g(\bar X_n)$ is ensured by that of the series  $\sum_n
m^{-\frac{\varepsilon'n}{p}}\frac{m_n(p+\varepsilon)}{m_n^{p+\varepsilon}}$.
Applying Theorem \ref{CRP2.3} still yields the desired result.
\end{proof}

\begin{proof}[Proof of Theorem \ref{CRT1.2}]
Let $a_n=n^{\alpha}$, and define
$g(x)=\left(\log x\right)^{1+\alpha+\varepsilon}$ for $x\geq c$, and $g(x) = g(c)$ for $x\in [0,c]$, where
 $c>0$ is taken large enough such that $g(c) >0$ and that $x \mapsto \frac{x}{g(x)}$ is increasing on $[c,\infty)$.
 Then $x\mapsto  g(x)$ and $x\mapsto \frac{x}{g(x)}$ are increasing on $[0,+\infty)$. The condition that $\liminf_n\frac{1}{n}\log \Pi_n=a>0$ implies that
$$\liminf_{n\rightarrow\infty} \frac{1}{n}\log\frac{\Pi_n}{n^{\alpha}}=a;$$
in particular, $\frac{\Pi_n}{a_n} \rightarrow +\infty$.
Therefore, for $n$ large enough, we have  $\log  \frac{\Pi_n}{n^{\alpha}}>\frac{an}{2}$, and
\begin{eqnarray*}
\frac{a_n} {g\left(\frac{\Pi_n}{a_n}\right)}
= \frac{n^\alpha}{ (\log  \frac{\Pi_n}{n^{\alpha}})^{1+\alpha+\varepsilon} }
\leq  \frac{n^\alpha}{(\frac{an}{2})^{1+\alpha+\varepsilon}}
=  \frac{ \big(\frac{2}{a}\big)^{1+\alpha+\varepsilon} }{n^{1+\varepsilon}}.
\end{eqnarray*}
On the other hand, by considering $\bar X_n \leq c$ and $\bar X_n >c$, we see that
$$ \mathbb E \bar X_n g(\bar X_n) \leq  cg(c) + \mathbb E \bar X_n (\log^+ \bar X_n)^{1+\alpha+\varepsilon}.$$
Thus   (\ref{log_moment}) implies (\ref{CRPE}) with $a_n$ and $g$ defined as above, and
 the result follows from Theorem \ref{CRP2.3} and Lemma \ref{CRL1.2.2}.
\end{proof}

\medskip
It remains the proof of Theorem \ref{CRP2.3}. The proof  will be based on a refinement of
an argument of  Asmussen (1976, \cite{soren}). The crucial idea is to find an appropriate truncation to show the convergence of the series $\sum_n a_n(W_{n+1}-W_n)$, which implies  the  convergence rate of $W-W_n$ by means of the following lemma.

\begin{lem}[\cite{soren},  Lemma 2]\label{CRL1.2.2}
Let $(a_n)$ be a sequence of real numbers. If $a_n\geq0$,
$a_n\uparrow\infty $, then the a.s. convergence of the series $\sum_n a_n(W_{n+1}-W_n)$ implies that $W-W_n=o(a_n^{-1})\;a.s.$.
Moreover,  for
$\alpha\geq1$, the a.s. convergence of the series $\sum_n n^\alpha(W_{n+1}-W_n)$  implies that  of $\sum_n
(W -W_n)$.
\end{lem}

\bigskip
\begin{proof}[Proof of Theorem \ref{CRP2.3}]Let
\begin{eqnarray*}
\bar{X}_{n,i}&=&\frac{X_{n,i}}{m_n},\qquad\bar{X}'_{n,i}=\bar{X}_{n,i}\mathbf{1}_{\{\bar{X}_{n,i}\leq \frac{\Pi_n}{a_n}\}},\\
S_n&=&a_n(W_{n+1}-W_n)=\frac{a_n}{\Pi_n}\sum_{i=1}^{Z_n}(\bar{X}_{n,i}-1),\\
S'_n&=&\frac{a_n}{\Pi_n}\sum_{i=1}^{Z_n}(\bar{X}'_{n,i}-1).
\end{eqnarray*}
Observe that
\begin{equation}\label{CRE1.2.1TP}
\sum_{n=0}^\infty S_n=\sum_{n=0}^\infty
(S_n-S_n')+\sum_{n=0}^\infty(S_n'-\mathbb{E}(S_n'\mid\mathcal
{F}_{n}))+\sum_{n=0}^\infty \mathbb{E}(S_n'\mid\mathcal {F}_{n}).
\end{equation}
We shall prove that each of  the three series on the right hand-side   converges a.s., so that $\sum_{n=0}^\infty S_n$ converges a.s. and the result follows from Lemma \ref{CRL1.2.2}.  According to Asmussen (1976, \cite{soren}), it suffices to show that
\begin{equation}\label{CRE1.2.1}
\sum_{n=0}^\infty \mathbb{P}(S_n\neq S_n')<\infty,\sum_{n=0}^\infty
\mathbb{E}(S_n'-\mathbb{E}(S_n'\mid\mathcal{F}_{n}))^2<\infty\;
\text{and}\;0\leq-\sum_{n=0}^\infty \mathbb{E}S'_n<\infty.
\end{equation}
Indeed, $\sum_n \mathbb{P}(S_n\neq S_n')<\infty$ implies the a.s. convergence of the series $\sum_n(S_n-S_n')$ by the Borel-Cantelli lemma;  $\sum_n
\mathbb{E}(S_n'-\mathbb{E}(S_n'\mid\mathcal{F}_{n}))^2<\infty$ means that the martingale $\{(A_n, \mathcal
{F}_{n+1}) \}$, with $A_n=\sum_{k=0}^n(S_k'-\mathbb{E}(S_k'\mid\mathcal {F}_{k}))$, is ${L}^2$-bounded, so that it converges  a.s., meaning that the series
$\sum_n(S_n'-\mathbb{E}(S_n'\mid\mathcal {F}_{n})) $ converges a.s.;  the series $\sum_n \mathbb{E}(S_n'\mid\mathcal {F}_{n})$ converges a.s. since
$ - \mathbb{E}(S_n'\mid\mathcal {F}_{n}) = a_n \frac{Z_n}{\Pi_n} \mathbb{E} \bar{X}_{n,i}\mathbf{1}_{\{\bar{X}_{n,i} > \frac{\Pi_n}{a_n}\}} \geq 0 $ (using $ \mathbb{E} \bar{X}_{n,i} =1$) a.s. and $0\leq-\sum_n \mathbb{E}S'_n<\infty$.

For the first series in  (\ref{CRE1.2.1}),
since
$$\left\{S_n\neq S_n'\right\}\subset \bigcup_{i=1}^{Z_n}\left\{\bar{X}_{n,i}> \frac{\Pi_n}{a_n}\right\},$$
and $x\mapsto xg(x)$ is  increasing  on $[0,+\infty)$ (so that
$\bar{X}_{n}> \frac{\Pi_n}{a_n}$ implies  $\bar{X}_{n}g(\bar{X}_{n})  \geq  \frac{\Pi_n}{a_n} g\left(\frac{\Pi_n}{a_n}\right)$), using Markov's inequality we obtain
\begin{eqnarray}\label{CRPP1}
 \mathbb{P}(S_n\neq S_n')&\leq&\
\mathbb{E}\sum_{i=1}^{Z_n}\mathbf{1}_{\{\bar{X}_{n,i}> \frac{\Pi_n}{a_n}\}}\nonumber\\
&=&
\Pi_n\mathbb{P}\left(\bar{X}_{n}> \frac{\Pi_n}{a_n}\right)\nonumber\\
&\leq &
\Pi_n\mathbb{P}\left(\bar{X}_{n} g(\bar{X}_{n})   \geq  \frac{\Pi_n}{a_n}g\left(\frac{\Pi_n}{a_n}\right)  \right)\nonumber\\
&\leq& \frac{a_n}{g\left(\frac{\Pi_n}{a_n}\right)}\mathbb E \bar X_n g(\bar X_n).
\end{eqnarray}

For the second series in  (\ref{CRE1.2.1}),
 since  $x\mapsto \frac{x}{g(x)}$ is increasing  on $[0,+\infty)$,
$\bar{X}_{n}\leq \frac{\Pi_n}{a_n}$ implies  $\frac{\bar{X}_{n}}{g(\bar{X}_{n})}
\leq  \frac{ \frac{\Pi_n}{a_n}}  {g(\frac{\Pi_n}{a_n})}$, so that
\begin{eqnarray}\label{CRPP2}
\mathbb{E}(S_n'-\mathbb{E}(S_n'\mid\mathcal{F}_{n}))^2
&=& \mathbb{E}\mbox{Var}(S'_n|\mathcal{F}_{n}) =\frac{a_n^2}{\Pi_n}\mbox{Var}\bar{X}'_n\nonumber\\
&\leq& \frac{a_n^2}{\Pi_n}
       \mathbb{E}\bar{X}_n^2 \mathbf{1}_{\{\bar{X}_{n}\leq\frac{\Pi_n}{a_n}\}}\nonumber\\
&=&     \frac{a_n^2}{\Pi_n}
     \mathbb{E}\bar{X}_n g(\bar{X}_n) \frac{\bar{X}_n}{g(\bar{X}_n)}
               \mathbf{1}_{\{\bar{X}_{n} \leq \frac{\Pi_n}{a_n}\}}\nonumber\\
&\leq&  \frac{a_n^2}{\Pi_n}   \mathbb E \bar X_n g(\bar X_n)
                                      \frac{\frac{\Pi_n}{a_n} }{g(\frac{\Pi_n}{a_n})} \nonumber\\
&= &  \frac{a_n}{g\left(\frac{\Pi_n}{a_n}\right)} \mathbb E \bar X_n g(\bar X_n).
\end{eqnarray}

For the last series in  (\ref{CRE1.2.1}),  recall that
$ \mathbb{E}(S_n'\mid\mathcal{F}_n)
=-\frac{a_n}{\Pi_n}Z_n\mathbb{E}\bar{X}_n
\mathbf{1}_{\{\bar{X}_{n}>\frac{\Pi_n}{a_n}\}}\leq0.
$
Therefore, since $g(x)$ is  increasing  on $[0,+\infty)$, we have
\begin{eqnarray}\label{CRPP3}
0\leq-\mathbb{E}S'_n
= {a_n}\mathbb{E}\bar{X}_n
\mathbf{1}_{\{\bar{X}_{n}>\frac{\Pi_n}{a_n}\}}
\leq \frac{a_n}{g\left(\frac{\Pi_n}{a_n}\right)}\mathbb E \bar X_n g(\bar X_n).
\end{eqnarray}
 By (\ref{CRPP1}), (\ref{CRPP2}) and (\ref{CRPP3}),  we see that (\ref{CRPE}) implies (\ref{CRE1.2.1}). This completes the proof.
\end{proof}

\section{ Proofs of Theorems \ref{CRT2.1.1} and \ref{CRT2.1.2}}\label{CR1S2}
For a branching process $(Z_n)$ in a random environment, when the environment $\xi$ is fixed,  it is a branching process in a varying environment. So we can directly apply the results for a branching  process in a varying environment to $(Z_n)$ by considering the conditional probability $\mathbb{P}_\xi$ and the corresponding expectation $\mathbb{E}_\xi$.

 \begin{lem}[\cite{grin}, Theorem 1]\label{CRL2.2.1}
 Let $(\alpha_n,\beta_n)_{n\geq0}$ be a stationary and ergodic sequence of nonnegative random variables.
If  $\mathbb{E}\log\alpha_0<0$  and $\mathbb{E}\log^+\beta_0<\infty$,
 then
\begin{equation}\label{CRE2.2.1}
\sum_{n=0}^{\infty}\alpha_0\cdots\alpha_{n-1}\beta_n<\infty\quad a.s..
\end{equation}
\end{lem}
To see the conclusion of Lemma \ref{CRL2.2.1}, it suffices to notice that by the ergodic theorem, under the above moment conditions, we have
$$\limsup_{n\rightarrow\infty}\ (\alpha_0\cdots\alpha_{n-1}\beta_n)^{1/n}<1  \quad a.s..$$
 \\

\begin{proof}[Proof of Theorem \ref{CRT2.1.1}]
 We shall prove Theorem \ref{CRT2.1.1} by applying Corollary \ref{CRT1.1}.
By the ergodic theorem,
$$\lim_{n\rightarrow\infty}\frac{\log \Pi_n}{n}=\mathbb{E}\log m_0>0\qquad a.s..$$
Let $q_1=q+\varepsilon$ and $1/p_1+1/q_1=1$. Take $\varepsilon_1=p-p_1\in(0,2-p_1)$ and
$0<\varepsilon'_1<\varepsilon_1$. We shall apply Lemma
\ref{CRL2.2.1} to the special case where  $(\alpha_n)$ is  deterministic:
$\alpha_n=m^{-\frac{\varepsilon_1'}{p_1}}$,  and
$\beta_n=\mathbb{E}_{\xi}\bar{X}_n^{p_1+\varepsilon_1}=\frac{m_n(p )}{m_n^{p }}$. Recall that $m=\exp(\mathbb{E}\log m_0)\in(1,\infty)$, so that
$\log \alpha_0\in(-\infty,0)$ ; obviously, $\mathbb{E}\log^+\beta_0=\mathbb{E}\log
\mathbb{E}_{\xi}(\frac{Z_1}{m_0})^{p }<\infty$. Therefore by Lemma
\ref{CRL2.2.1} , $\sum_{n=0}^\infty
m^{-\frac{\varepsilon'_1n}{p_1}}\frac{m_n(p_1+\varepsilon_1)}{m_n^{p_1+\varepsilon_1}}<\infty$ a.s.. So Theorem \ref{CRT2.1.1} is a direct consequence of Theorem \ref{CRT1.1}(ii).

\end{proof}

\begin{proof}[Proof of Theorem \ref{CRT2.1.2}] Notice that under the stronger moment condition $\mathbb{E}\frac{Z_1}{m_0}(\log^+\frac{Z_1}{m_0})^{1+\alpha + \varepsilon }<\infty$ for some $\varepsilon >0$, the conclusion in Theorem \ref{CRT2.1.2} is a direct consequence of Theorem \ref{CRT1.2}.
Instead of using Theorem \ref{CRT1.2} directly, to obtain more precise result assuming only $\mathbb{E}\frac{Z_1}{m_0}(\log^+\frac{Z_1}{m_0})^{1+\alpha  }<\infty$, we choose another truncation
function.
Set $\bar{X}'_{n,i}$, $S_n$, $S'_n$ like in the proof of
Theorem \ref{CRP2.3} but with the truncation function
$\mathbf{1}_{\{\bar{X}_{n,i}(\log\bar{X}_{n,i})^{\kappa}\leq
\frac{\Pi_n}{n^{\alpha}}\}}$ in place of
$\mathbf{1}_{\{\bar{X}_{n,i}\leq \frac{\Pi_n}{a_n}\}}$. The
value of $\kappa>0$ will be taken suitably large. It suffices to
prove the $\mathbb{P}_{\xi}$-$a.s.$ convergence  of the following three
series :
$$\sum_{n=0}^\infty
(S_n-S_n'),\qquad\sum_{n=0}^\infty(S_n'-\mathbb{E}_{\xi}(S_n'\mid\mathcal
{F}_{n})),\qquad\sum_{n=0}^\infty \mathbb{E}_{\xi}(S_n'\mid\mathcal {F}_{n}).$$
By a simple calculation, we have
$$0\leq
S_n-S'_n=\frac{n^{\alpha}}{\Pi_n}\sum_{i=1}^{Z_n}\bar{X}_{n,i}\mathbf{1}_{\{\bar{X}_{n,i}(\log\bar{X}_{n,i})^{\kappa}>
\frac{\Pi_n}{n^{\alpha}}\}},$$
$$\mathbb{E}_{\xi}(S_n-S'_n)=n^{\alpha}\mathbb{E}_{\xi}\bar{X}_n\mathbf{1}_{\{\bar{X}_{n}(\log\bar{X}_{n})^{\kappa}>
\frac{\Pi_n}{n^{\alpha}}\}},$$ and
$$0\leq
-\mathbb{E}_{\xi}S'_n=n^{\alpha}\mathbb{E}_{\xi}\bar{X}_n\mathbf{1}_{\{\bar{X}_{n}(\log\bar{X}_{n})^{\kappa}>
\frac{\Pi_n}{n^{\alpha}}\}}.$$ By the ergodic theorem,
\begin{equation}\label{CRE+1}
\lim_{n\rightarrow\infty}\frac{1}{n}\log\frac{\Pi_n}{n^\alpha}=\mathbb{E}\log
m_0>0\qquad a.s..
\end{equation}
So for some constant $c_1>1$, we have $\Pi_n/n^{\alpha}\geq c_1^n$  for $n$ large enough,
therefore,
$$
n^{\alpha}\mathbb{E}_{\xi}\bar{X}_n\mathbf{1}_{\{\bar{X}_{n}(\log\bar{X}_{n})^{\kappa}>
\frac{\Pi_n}{n^{\alpha}}\}}\leq
n^{\alpha}\mathbb{E}_{\xi}\bar{X}_n\mathbf{1}_{\{\bar{X}_{n}(\log\bar{X}_{n})^{\kappa}>
c_1^n\}}.
$$
Notice that
\begin{eqnarray*}
&&\mathbb{E}\left(\sum_{n=0}^{\infty}n^{\alpha}\mathbb{E}_{\xi}\bar{X}_n\mathbf{1}_{\{\bar{X}_{n}(\log\bar{X}_{n})^{\kappa}>
c_1^n\}}\right)\\
&=&\sum_{n=0}^{\infty}n^{\alpha}\mathbb{E}\bar{X}_0\mathbf{1}_{\{\bar{X}_{0}(\log\bar{X}_{0})^{\kappa}>
c_1^n\}}\\
&=&\mathbb{E}\bar{X}_0\sum_{n=0}^{\infty}n^{\alpha}\mathbf{1}_{\{\bar{X}_{0}(\log\bar{X}_{0})^{\kappa}>
c_1^n\}}\\
&=&\mathbb{E}\bar{X}_0\sum_{n=0}^{n_0}n^{\alpha}\leq\mathbb{E}\bar{X}_0{n_0}^{\alpha+1},
\end{eqnarray*}
where $$n_0=\max\left\{0, \left\lceil (\log c_1)^{-1}(\log \bar{X}_0+\kappa\log\log \bar{X}_0 )-1\right\rceil\right\}$$
with the notation $\left\lceil a\right\rceil=\min\{k\in\mathbb Z: k\geq a\}$, and
$$\mathbb{E}\bar{X}_0{n_0}^{\alpha+1}\leq  C \mathbb{E}\bar{X}_0(\log^+\bar{X}_0)^{\alpha+1}<\infty, $$
where $C$ is a positive constant (suitably large). Hence it follows
the $\mathbb{P}_{\xi}$-a.s. convergences of $\sum_n (S_n-S_n')$
and $\sum_n \mathbb{E}_{\xi}(S_n'\mid\mathcal {F}_{n})$.
Now we consider the series
$\sum_n(S_n'-\mathbb{E}_{\xi}(S_n'\mid\mathcal {F}_{n}))$. We have
$$\mathbb{E}_{\xi}(S_n'-\mathbb{E}_{\xi}(S_n'\mid\mathcal
{F}_{n}))^2\leq\frac{n^{2\alpha}}{\Pi_n}\mathbb{E}_{\xi}\bar{X}^2_n\mathbf{1}_{\{\bar{X}_{n}(\log\bar{X}_{n})^{\kappa}\leq
\frac{\Pi_n}{n^{\alpha}}\}}.$$
 Take $f(x)=x(\log
x)^{-(\alpha+1+\delta)}\;(\delta>0)$. It is clear that $f(x)$ is increasing and
positive on $(c,+\infty)$, where $c>e$ is a suitably large
constant. For $n$ large enough,
\begin{eqnarray*}
&&\frac{n^{2\alpha}}{\Pi_n}\mathbb{E}_{\xi}\bar{X}^2_n\mathbf{1}_{\{\bar{X}_{n}(\log\bar{X}_{n})^{\kappa}\leq
\frac{\Pi_n}{n^{\alpha}}\}}\\
&=&\frac{n^{2\alpha}}{\Pi_n}\mathbb{E}_{\xi}\bar{X}^2_n\mathbf{1}_{\{\bar{X}_{n}(\log\bar{X}_{n})^{\kappa}\leq
\frac{\Pi_n}{n^{\alpha}}\}}\mathbf{1}_{\{\bar{X}_n\geq
c\}}+\frac{n^{2\alpha}}{\Pi_n}\mathbb{E}_{\xi}\bar{X}^2_n\mathbf{1}_{\{\bar{X}_{n}(\log\bar{X}_{n})^{\kappa}\leq
\frac{\Pi_n}{n^{\alpha}}\}}\mathbf{1}_{\{\bar{X}_n< c\}}\\
&\leq&\frac{n^{2\alpha}}{\Pi_n}\mathbb{E}_{\xi}\bar{X}^2_n\frac{f(\frac{\Pi_n}{n^{\alpha}})}{f(\bar{X}_n(\log\bar{X}_n)^{\kappa})}\mathbf{1}_{\{\bar{X}_n\geq
c\}}+c^2\frac{n^{2\alpha}}{\Pi_n}\\
&=&\frac{n^\alpha}{(\log \frac{\Pi_n}{n^\alpha} )^{\alpha+1+\delta}}\mathbb{E}_\xi\frac{\bar X_n(\log X_n+\kappa\log\log \bar X_n)^{\alpha+1+\delta}}{(\log \bar X_n)^\kappa}\mathbf{1}_{\{\bar{X}_n\geq
c\}}+c^2\frac{n^{2\alpha}}{\Pi_n}\\
&\leq&C\frac{1}{n^{1+\delta}}\mathbb{E}_\xi\bar{X}_n(\log \bar{X}_n)^{\alpha+1+\delta-\kappa}\mathbf{1}_{\{\bar{X}_n\geq
c\}}+c^2\frac{n^{2\alpha}}{\Pi_n},
\end{eqnarray*}
where the last step holds by (\ref{CRE+1}) and the fact that $\lim_{x\rightarrow\infty}\frac{(\log x+\kappa\log \log x)^{\alpha+1+\delta}}{(\log x)^{\alpha+1+\delta}}=1$.
It is easy to see that
$\sum_{n=1}^{\infty}\frac{n^{2\alpha}}{\Pi_n}<\infty$ a.s.. Besides,
taking $\kappa\geq\delta$, we obtain
\begin{eqnarray*}&&\mathbb{E}\left(\sum_{n=1}^{\infty}\frac{1}{n^{1+\delta}}\mathbb{E}_{\xi}\bar{X}_n(\log\bar{X}_n)^{\alpha+1+\delta-\kappa}\mathbf{1}_{\{\bar{X}_n\geq
c\}}\right)\\
&=&\mathbb{E}\bar{X}_0(\log\bar{X}_0)^{\alpha+1+\delta-\kappa}\mathbf{1}_{\{\bar{X}_0\geq
c\}}\sum_{n=1}^{\infty}\frac{1}{n^{1+\delta}}\\
&\leq&C
\mathbb{E}\bar{X}_0(\log^+\bar{X}_0)^{\alpha+1}\sum_{n=1}^{\infty}\frac{1}{n^{1+\delta}}<\infty.
\end{eqnarray*}
Hence $\sum_n\frac{1}{n^{1+\delta}}\mathbb{E}_{\xi}\bar{X}_n(\log\bar{X}_n)^{\alpha+1+\delta-\kappa}\mathbf{1}_{\{\bar{X}_n\geq
c\}}<\infty$ $\mathbb{P}_{\xi}$-a.s..
 Thus
$\sum_n(S_n'-\mathbb{E}_{\xi}(S_n'\mid\mathcal {F}_{n}))$
converges $\mathbb{P}_{\xi}$-a.s.. The proof is completed.

\end{proof}

\noindent\textbf{Remark.} 
 The condition $\mathbb{E}\frac{Z_1}{m_0}(\log^+\frac{Z_1}{m_0})^{1+\alpha}<\infty$ can be replaced by
$\mathbb{E}\frac{Z_1}{m_0}(\log^+ {Z_1} )^{1+\alpha}<\infty$. To see this, one can repeat the above proof  by considering the truncation function $\mathbf{1}_{\{{X}_{n,i}(\log{X}_{n,i})^{\kappa}\leq
\frac{\Pi_{n+1}}{n^{\alpha}}\}}$.

\section{Proof of Theorem \ref{CRASTP}}\label{CR1S3}
Notice that
\begin{equation}\label{CRASE5.1}
W-W_n=\frac{1}{\Pi_n}\sum_{j=1}^{Z_n}(W(n,j)-1),
\end{equation}
where $W(n,j)$ ($j=1,2,\cdots$) denotes the limit random variable in the line of
descent initiated by the $jth$ particle of the $nth$ generation.
Under $\mathbb{P}_\xi$,  $(W(n,j))_{j\geq 1}$ are
independent of each other and independent of $\mathcal{F}_n$, and have the same distribution
$\mathbb{P}_{\xi}(W(n,j)\in\cdot)=\mathbb{P}_{T^n\xi}(W\in\cdot)$.

 If $\delta_\infty(\xi)\in(0,\infty)$ a.s., let
\begin{equation}\label{CRASE5.2}
V_{n,j}=\frac{W(n,j)-1}{\delta_\infty(T^n\xi)}.
\end{equation}
Then $\mathbb{E}_\xi V_{n,j}=0$ and $\mbox{Var}_\xi V_{n,j}=1$. From (\ref{CRASE5.1}), we have
\begin{equation}\label{CRASE5.3}
\frac{\sqrt{\Pi_n}(W-W_n)}{\delta_\infty(T^n\xi)}=\frac{1}{\sqrt{Z_n}}\sum_{j=1}^{Z_n}V_{n,j}\cdot\sqrt{W_n}.
\end{equation}

\begin{lem}\label{CRASL5.1}
Let $(r_n)\subset\mathbb{N}$ be a sequence of positive integers such that $r_n\rightarrow\infty$ as $n\rightarrow\infty$. Assume that  $\delta_\infty(\xi)\in(0,\infty)$ a.s. and set $F_n(r,x)=\mathbb{P}_\xi\left(\frac{1}{\sqrt{r}}\sum_{j=1}^{r}V_{n,j}\leq x\right)$. Then
$$\sup_{x\in\mathbb{R}}|F_n(r_n,x)-\Phi(x)|\rightarrow0\quad\text{in $L^1$,}$$
where $\Phi(x)=\frac{1}{\sqrt{2\pi}}\int_{-\infty}^x e^{-t^2/2}dt$ is the distribution function of standard normal distribution $\mathcal{N}(0,1)$.
\end{lem}
\begin{proof}
By the stationarity of the environment sequence, we have
\begin{eqnarray*}
 \mathbb{E}\sup_{x\in\mathbb{R}}|F_n(r_n,x)-\Phi(x)|
&=&\mathbb{E}\sup_{x\in\mathbb{R}}|\mathbb{P}_\xi( \sum_{j=1}^{r_n}V_{n,j}\leq \sqrt{r_n}x)-\Phi(x)|\\
&=&\mathbb{E}\sup_{x\in\mathbb{R}}|\mathbb{P}_\xi( \sum_{j=1}^{r_n}V_{1,j}\leq \sqrt{r_n}x)-\Phi(x)|.
\end{eqnarray*}
Notice that $(V_{1,j})_{j\geq1}$ are i.i.d. under $\mathbb{P_\xi}$. By the classic central limit theorem,
$$\sup_{x\in\mathbb{R}}|\mathbb{P}_\xi( \sum_{j=1}^{r_n}V_{1,j}\leq \sqrt{r_n}x)-\Phi(x)|\rightarrow0\quad a.s..$$
So the dominated convergence theorem ensures that
$$\mathbb{E}\sup_{x\in\mathbb{R}}|\mathbb{P}_\xi( \sum_{j=1}^{r_n}V_{1,j}\leq \sqrt{r_n}x)-\Phi(x)|\rightarrow0.$$
Therefore
$$\mathbb{E}\sup_{x\in\mathbb{R}}|\mathbb{P}_\xi(Y_n\leq x)-\Phi(x)|\rightarrow0.$$

\end{proof}

\begin{proof}[Proof of Theorem \ref{CRASTP}]
Let $U_n=\frac{\sqrt{\Pi_n}(W-W_n)}{\delta_\infty(T^n\xi)}$. Notice that  (\ref{CRASEP1})  implies (\ref{CRASEP2}) since
\begin{eqnarray*}
\sup_{x\in\mathbb{R}}|\mathbb{P}(U_n\leq x)-\Phi_2(x)|&=&\sup_{x\in\mathbb{R}}|\mathbb{E}[\mathbb{P}_\xi(U_n\leq x)]-\mathbb{E}\Phi_1(x)|\\
&\leq&\mathbb{E}\sup_{x\in\mathbb{R}}|\mathbb{P}_\xi(U_n\leq x)-\Phi_1(x)|.
\end{eqnarray*}
So we only need  to prove (\ref{CRASEP1}). Since $\mathbb{E}\log\mathbb{E}_\xi\left(\frac{Z_1}{m_0}\right)^2<\infty$, we have the a.s. convergence of the series $\sum_{n=0}^\infty\frac{1}{\Pi_n}\left(\frac{m_n(2)}{m_n^2}-1\right)$ by Lemma \ref{CRL2.2.1}. According to the statements before Theorem \ref{CRASTP}, the martingale $\{W_n\}$ is $L^2$-bounded under $\mathbb{P}_\xi$ for almost all $\xi$, so that $W_n\rightarrow W$ in $L^2$ under $\mathbb{P}_\xi$. Therefore, $\mathbb E_\xi W=1$ a.s., which implies that $\mathbb{P}_\xi(W=0)<1$ a.s.. Hence the extinction probability $\mathbb{P}_\xi(Z_n\rightarrow0)\leq \mathbb{P}_\xi(W=0)<1$ a.s., since the set $\{Z_n\rightarrow0\}$ is a.s. contained in $\{W=0\}$. Notice that $\mathbb{P}_\xi(W=0)$ is a solution of the functional equation
$$q(\xi)=\varphi_{\xi_0}(q(T\xi)).$$
By the uniqueness of the solution (Theorem 6 of \cite{a12}),
\begin{equation}\label{CRE+2.1}
\mathbb{P}_\xi(W=0)=\mathbb{P}_\xi(Z_n\rightarrow0)\qquad a.s..
\end{equation}
Now
$$\mathbb{P}_\xi(U_n\leq x)=\mathbb{P}_\xi(U_n\leq x, Z_n\rightarrow\infty)+\mathbb{P}_\xi(U_n\leq x, Z_n\rightarrow0).$$
Notice that $U_n\rightarrow0$ a.s. on $\{Z_n\rightarrow0\}$. Thus
\begin{equation}\label{CRE+3}
\mathbb{P}_\xi(U_n\leq x, Z_n\rightarrow0)\rightarrow\mathbf{1}_{\{0\leq x\}}\mathbb{P}_\xi(Z_n\rightarrow0)=\mathbf{1}_{\{0\leq x\}}\mathbb{P}_\xi(W=0)\qquad a.s..
\end{equation}
On the other side, by (\ref{CRASE5.3}),
$$
\mathbb{P}_\xi(U_n\leq x, Z_n\rightarrow\infty) =\mathbb{P}_\xi\left(\frac{1}{\sqrt{Z_n}}\sum_{j=1}^{Z_n}V_{n,j}\cdot\sqrt{W_n}\leq x, Z_n\rightarrow\infty\right).
$$
As $Z_n$ is independent of $(V_{n,j})_{j\geq1}$ under $\mathbb{P}_\xi$, it follows that
\begin{equation}\label{CRASE5.5}
 \mathbb{P}_\xi(U_n\leq x, Z_n\rightarrow\infty)=\mathbb{E}_\xi F_n(Z_n, x/\sqrt{W_n})\mathbf{1}_{\{Z_n\rightarrow\infty\}}.
  \end{equation}
   By Lemma \ref{CRASL5.1}, for each sequence $(n')$ of $\mathbb{N}$ with $n'\rightarrow\infty$, there exists a subsequence $(n'')$ of $(n')$ such that $n''\rightarrow\infty$ and
 \begin{equation}\label{CRE+2.2}
  \sup_{x\in\mathbb{R}}|F_{n''}(r_{n''},x)-\Phi(x)|\rightarrow0, \quad\text{whenever $r_n\rightarrow\infty$.}
\end{equation}
  Take $r_n=Z_n$. By (\ref{CRASE5.5}), (\ref{CRE+2.1}), (\ref{CRE+2.2}) and the dominated convergence theorem, we see that for each $x\in\mathbb{R}$, a.s.,
  \begin{eqnarray*}
  &&\left| \mathbb{P}_\xi(U_{n''}\leq x, Z_n\rightarrow\infty)-\mathbb{E}_\xi\Phi(x/\sqrt{W})\mathbf{1}_{\{W>0\}}\right|\\
  &=&\left|\mathbb{E}_\xi\left(F_{n''}(Z_{n''}, x/\sqrt{W_{n''}})-\Phi(x/\sqrt{W_{n''}})\right)\mathbf{1}_{\{Z_n\rightarrow\infty\}}+
  \mathbb{E}_\xi\Phi(x/\sqrt{W_{n''}})\left(\mathbf{1}_{\{Z_n\rightarrow\infty\}}-\mathbf{1}_{\{W>0\}}\right)\right.\\
  &&\left.+\mathbb{E}_\xi\left(\Phi(x/\sqrt{W_{n''}})-\Phi(x/\sqrt{W})\right)\mathbf{1}_{\{W>0\}}\right|\\
  &\leq&\mathbb{E}_\xi\sup_{x\in\mathbb{R}}|F_{n''}(Z_{n''}, x)-\Phi(x)|\mathbf{1}_{\{Z_n\rightarrow\infty\}}+\mathbb{E}_\xi\left|\Phi(x/\sqrt{W_{n''}})-\Phi(x/\sqrt{W})\right|\mathbf{1}_{\{W>0\}}\rightarrow0.
  \end{eqnarray*}
Notice that $\Phi_1(x)=\mathbb{E}_\xi\Phi({x}/{\sqrt{W}})\mathbf{1}_{\{W>0\}}+\mathbb{P}_\xi(W=0)\mathbf{1}_{\{x\geq0\}}$.
Therefore, we have for each $x\in\mathbb{R}$, a.s.,
 \begin{equation}\label{CRASE5.6}
  \mathbb{P}_\xi(U_{{n''}}\leq x)\rightarrow \Phi_1(x),
 \end{equation}
 where the null set may depend on $x$.
 Thus a.s. (\ref{CRASE5.6}) holds for all rational $x$, and then for all $x\in\mathbb{R}$ by the monotonicity of the left term and the continuity of the right term, where the null set is independent of $x$. Hence by Dini's theorem,
  $$\sup_{x\in\mathbb{R}}|\mathbb{P}_\xi(U_{n''}\leq x)- \Phi_1(x)|\rightarrow0\qquad a.s..$$
By the dominated convergence theorem,
\begin{equation}\label{CRASE5.7}
  \mathbb{E}\sup_{x\in\mathbb{R}}|\mathbb{P}_\xi(U_{n''}\leq x)- \Phi_1(x)|\rightarrow0.
 \end{equation}
 So we have proved that for each subsequence $(n')$ of $\mathbb{N}$ with $n'\rightarrow\infty$, there is a subsequence $(n'')$ of $(n')$ with $n''\rightarrow\infty$ such that (\ref{CRASE5.7}) holds. Thus (\ref{CRASEP1}) holds.
\end{proof}

\section{ Proof of Theorem \ref{CRT2.1.9}}\label{CR1S4}
In this section we consider a branching process $(Z_n)$ in a finite state random environment, where each
$\xi_n$ takes values in a finite set $\{a_1,a_2,\cdots,a_N\}$.  Before the proof of
Theorem \ref{CRT2.1.9}, we first prove  three lemmas.
\\*

The first lemma is an elementary result about the Laplace transform.

\begin{lem}\label{CRL2.5.1}
Let $X$ be a random variable with  $\mathbb{E}X=0$. If $\mathbb{E}e^{\delta|X|}\leq
K$  for some constants $\delta>0$ and $K>0$, then $\mathbb{E}e^{tX}\leq e^{K_{\delta}t^2}$ for
$t\in (0,\frac{\delta}{2})$, where $K_{\delta}=\frac{2K}{\delta^2}$.
\end{lem}

\begin{proof}[\textbf{Proof}]
For $t\in(0,\frac{\delta}{2})$,
\begin{eqnarray*}
\mathbb{E}e^{tX}&=&\mathbb{E}\sum_{k=0}^\infty\frac{(tX)^k}{k!}=1+\sum_{k=2}^\infty(\frac{t}{\delta})^k\frac{\mathbb{E}(\delta
X)^k}{k!} \\&\leq&1+\sum_{k=2}^\infty(\frac{t}{\delta})^k\mathbb{E}e^{\delta
|X|}\leq1+K\frac{(t/\delta)^2}{1-t/\delta}\leq e^{K_{\delta}t^2}.
\end{eqnarray*}

\end{proof}

The second lemma is a generalization of a result of Athreya (1974, \cite{ath}, Theorem 4)  on the Galton-Watson process about the exponential moments of $W_n$; see also Liu (1996, \cite{liu96}, Theorem 2.1) for a slightly more complete result.

\begin{lem}\label{CRL2.5.2}
Let $(Z_n)$ be a branching process in a finite state random environment with
 $m_0>1$ a.s.. If  $\mathbb{E}_{\xi}e^{\theta_0W_1}\leq K\;a.s. $
for some constants $\theta_0>0$ and $K>0$,  then there
exist constants $\theta_1>0$ and $C_1>0$ such that
$$\sup_n\mathbb{E}_{\xi}e^{\theta_1W_n}\leq C_1\qquad a.s.. $$
\end{lem}

\begin{proof}
Assume that $\xi_0$ take values in a finite set $\{a_1, a_2, \cdots, a_N\}$.
Denote
$$\varphi_{\xi_n}(s)=\sum_{i=0}^\infty p_i(\xi_n)s^i,\;s\in[0,1]$$
the probability generating distribution of $p(\xi_n)$.
Let $\psi_n(t,\xi)=\mathbb{E}_{\xi}e^{tW_n}$, we have
\begin{equation}\label{CRE2.5.1}
\psi_{n+1}(t,\xi)=\varphi_{\xi_0}(\psi_n( \frac{t}{m_0},T\xi)).
\end{equation}
Noticing that $\mathbb{E}_{\xi}(W_1-1)=0$ and $\mathbb{E}_{\xi}e^{\theta_0|W_1-1|}\leq
Ke^{\theta_0}\;a.s. $, by Lemma \ref{CRL2.5.1}, we have for
$t\in(0,\frac{\theta_0}{2})$,
$$\psi_1(t,\xi)=\mathbb{E}_{\xi}e^{tW_1}=e^t\mathbb{E}_{\xi}e^{t(W_1-1)}\leq e^{t+K_1t^2}\qquad a.s., $$
where $K_1>\frac{2Ke^{\theta_0}}{\theta_0^2}$ is a suitable large constant.

Let
$g_{\xi_0}(t)=\varphi_{\xi_0}(e^{t/m_0+K_1t^2/m_0^2})e^{-t-K_1t^2}$.
It is not difficult to see that $g'_{\xi_0}(0)=0$ and
$g''_{\xi_0}(0)={\varphi_{\xi_0}''(1)}/{m_0^2}+(2K_1+1)( {1}/{m_0}-1)<0$ for $K_1$ large enough since $m_0>1$. Hence there exists
$t_0(\xi_0)>0$ such that $g_{\xi_0}(t)$ is decreasing on
$[0,t_0(\xi_0)]$. Notice that $\xi_0$ take value in
$\{a_1,a_2,\cdots,a_N\}$. We therefore can take $t_0=\min_{1\leq i\leq
N}\{t_0(a_i)\}>0$ such that $g_{\xi_0}(t)$ is decreasing on $[0,t_0]$,
thus $g_{\xi_0}(t)\leq g_{\xi_0}(0)=1$ on $[0,t_0]$.

Take $\theta_1=\min\{\frac{\theta_0}{2}, t_0\}$. From (\ref{CRE2.5.1}) and
considering the fact that $g_{\xi_0}(t)\leq1$ on $[0,\theta_1]$, by
induction we can get
$$\psi_n(t,\xi)\leq e^{t+K_1t^2}\qquad(\forall t\in(0, \theta_1))\qquad a.s..$$
Therefore the conclusion of the lemma follows with $C_1=e^{\theta_1+K_1\theta_1^2}$.

\end{proof}

The third lemma is a more general result than Theorem \ref{CRT2.1.9} about the supergeometric convergence rate of $\mathbb{P}_\xi(|W-W_n|>\varepsilon)$, where we do not assume that the environment has finite state space.

\begin{lem}\label{CRL2.5.3}
Suppose that $p_0(\xi_0)=0 \;a.s.$. If $\sup_n\mathbb{E}_{\xi}e^{\theta_1W_n}\leq
C_1\;a.s. $ for some constants $\theta_1>0$ and  $C_1>0$, then there exist
constants $\gamma>0$ and $C>0$ such that
$$\mathbb{P}_{\xi}(|W-W_n|>\varepsilon)\leq C\exp({-\gamma \Pi_n^{1/3}\varepsilon^{2/3}})\qquad a.s..$$
\end{lem}

\begin{proof}
We adopt the method in Athreya (1994, \cite{ath}) to prove Lemma \ref{CRL2.5.3}.
Recall  that
\begin{equation}\label{CRE2.5.2}
W-W_n=\frac{1}{\Pi_n}\sum_{j=1}^{Z_n}(W(n,j)-1).
\end{equation}
 Let $\phi_{\xi}(\theta)=\mathbb{E}_{\xi}e^{\theta W}$. It is clear that
$\phi_{\xi}(\theta)<\infty$ for $0<\theta\leq\theta_1$.
Denote
$$f_{\xi}(\theta,k)=\mathbb{E}_{\xi}\exp\left(\frac{\theta}{\sqrt{k}}\sum_{j=1}^k(W(n,j)-1)\right)\qquad(k\in\mathbb N).$$
For $0<\theta\leq\min\{\theta_1,1\}$, we have
\begin{eqnarray*}
f_{\xi}(\theta,k)=\left(\phi_{T^n\xi}(\frac{\theta}{\sqrt{k}})e^{-\theta/\sqrt{k}}\right)^k&=&
\left(1+\frac{1}{k}\frac{\phi_{T^n\xi}(\frac{\theta}{\sqrt{k}})e^{-\theta/\sqrt{k}}-1}{\theta^2/k}\theta^2\right)^
k\\&\leq&\exp\left(\frac{|\phi_{T^n\xi}(\frac{\theta}{\sqrt{k}})e^{-\theta/\sqrt{k}}-1|}{\theta^2/k}\right).
\end{eqnarray*}
Since for $t\in(0,\frac{\theta_1}{2})$,
\begin{eqnarray*} \frac{|\phi_{T^n\xi}(t)e^{-t}-1|}{t^2}&=&\left|\mathbb{E}_{T^n\xi}\sum_{k=2}^\infty
t^{k-2}\frac{(W-1)^k}{k!}\right|\\
&\leq&\sum_{k=2}^\infty(\frac{t}{\theta_1})^{k-2}\theta_1^{-2}\mathbb{E}_{T^n\xi}e^{\theta_1|W-1|}\\
&\leq&\sum_{k=2}^\infty(\frac{t}{\theta_1})^{k-2}\theta_1^{-2}e^{\theta_1}\mathbb{E}_{T^n\xi}e^{\theta_1W}\\
&\leq&\frac{C_1\theta_1^{-2}e^{\theta_1}}{1-t/{\theta_1}}<2C_1\theta_1^{-2}e^{\theta_1},
\end{eqnarray*}
it follows that for $0<\theta<\min\{\frac{\theta_1}{2},1\}$,
\begin{equation}\label{CRE+2}
f_{\xi}(\theta,k)\leq
\exp{( 2C_1\theta_1^{-2}e^{\theta_1})}=:C_2\qquad( \forall k).
\end{equation}
By (\ref{CRE2.5.2}) and (\ref{CRE+2}), we have for $0<\theta<\min\{\frac{\theta_1}{2},1\}$,
\begin{eqnarray*}
\mathbb{P}_{\xi}(W-W_n>\varepsilon)&=&\mathbb{P}_\xi\left(\frac{1}{\sqrt{Z_n}}\sum_{j=1}^{Z_n}(W_{n,j}-1)>\frac{\sqrt{\Pi_n}\varepsilon}{\sqrt{W_n}}\right)\\
&\leq&\mathbb{E}_{\xi}\frac{\mathbb{E}_{\xi}\left(\left.\exp(\frac{\theta}{\sqrt{Z_n}}\sum_{j=1}^{Z_n}(W(n,j)-1))\right|\mathcal
{F}_n\right)}{\exp({\theta \sqrt{\Pi_n}\varepsilon}/{\sqrt{W_n}})}\\
&=&\mathbb{E}_{\xi}f_{\xi}(\theta,Z_n)\exp \left(-\frac {\theta
\sqrt{\Pi_n}\varepsilon}{\sqrt{W_n}} \right)\\&\leq& C_2\mathbb{E}_{\xi}\exp \left(-\frac
{\theta \sqrt{\Pi_n}\varepsilon}{\sqrt{W_n}} \right).
\end{eqnarray*}
Using the identity
$ \mathbb{E}_{\xi} e^{-\lambda X} = \lambda \int_0^\infty e^{-\lambda u}\mathbb{P}_{\xi}( X\leq u) du$ for
$ X= 1/ \sqrt{W_n}$ and $\lambda>0$,   and the exponential Markov inequality $\mathbb{P}_{\xi}(W_n>1/u^2) \leq  \mathbb{E}_{\xi} (e^{\theta_1W_n})  e^{-\theta_1/u^2}$ ($u\neq 0$),
we  see that for every $\lambda>0$,
\begin{eqnarray*}
B_n(\lambda):= \mathbb{E}_{\xi}\exp(-\frac{\lambda}{\sqrt{W_n}})
&=&\lambda\int_0^\infty
e^{-\lambda u}\mathbb{P}_{\xi}(W_n>1/u^2)du\\
&\leq&\mathbb{E}_{\xi}e^{\theta_1W_n}\int_0^\infty
e^{-t-\theta_1\lambda^2/t^2}dt\\&\leq& C_1\int_0^\infty
e^{-t-\theta_1\lambda^2/t^2}dt\\
&=&C_1\left(\int_0^{\theta_1^{1/3}\lambda^{2/3}}e^{-t-\theta_1\lambda^2/t^2}dt+\int_{\theta_1^{1/3}\lambda^{2/3}}^\infty e^{-t-\theta_1\lambda^2/t^2}dt\right)\\
&\leq&C_1\left(e^{-\theta_1^{1/3}\lambda^{2/3}}\int_0^{\theta_1^{1/3}\lambda^{2/3}}e^{-t}dt+\int_{\theta_1^{1/3}\lambda^{2/3}}^\infty e^{-t}dt\right)\\
&=&C_1\left[e^{-\theta_1^{1/3}\lambda^{2/3}}(1-e^{-\theta_1^{1/3}\lambda^{2/3}})+e^{-\theta_1^{1/3}\lambda^{2/3}}\right]\\
&\leq&2C_1\exp(-\theta_1^{1/3}\lambda^{2/3}).
\end{eqnarray*}
Therefore, for $0<\theta_2<\min\{\frac{\theta_1}{2},1\}$
\begin{eqnarray*}
\mathbb{P}_{\xi}(W-W_n>\varepsilon)
&\leq&C_2B_n(\theta_2\sqrt{\Pi_n}\varepsilon)\\
&\leq&2C_1C_2\exp(-\theta_1^{1/3}\theta_2^{2/3}\Pi_n^{1/3}\varepsilon^{2/3})\\
&=&C\exp({-\gamma \Pi_n^{1/3}\varepsilon^{2/3}}),
\end{eqnarray*}
where $C=2C_1C_2>0$, $\gamma=\theta_1^{1/3}\theta_2^{2/3}>0$. For
$\mathbb{P}_{\xi}(W_n-W>\varepsilon)$, the argument is similar.

\end{proof}

\begin{proof}[Proof of Theorem \ref{CRT2.1.9}]
Notice that $\mathbb{E}_{\xi}e^{\theta_0W_1}$ depends only on $\xi_0$. Therefore, when $\xi_0$ has a finite state space, the following three conditions are equivalent:
$$  \mathbb{E}_{\xi}e^{\theta_0W_1}<\infty\;a.s.,\qquad   \mbox{ess sup}\ \mathbb{E}_{\xi}e^{\theta_0W_1}<\infty,\qquad   \mathbb{E}e^{\theta_0W_1}<\infty.$$
Moreover, notice that $m_0>1$ a.s. since $p_0(\xi_0)=0$ and $p_1(\xi_0)<1$ a.s.. By Lemma \ref{CRL2.5.2}, there exist
constants $\theta_1>0$ and $C_1>0$ such that
$\sup_n\mathbb{E}_{\xi}e^{\theta_1W_n}\leq C_1\; a.s..$
Hence part (a) is a direct consequence of Lemma  \ref{CRL2.5.2} and the equality (\ref{moment}) (with $\theta=\theta_1$).

For part (b),  by Lemmas  \ref{CRL2.5.2} and
 \ref{CRL2.5.3}, we see that  (\ref{CRE2.1.3}) holds. Taking the expectation in (\ref{CRE2.1.3}) and noticing the fact that $\Pi_n\geq \underline{m}^n$, we immediately
get (\ref{CRE2.1.4}).

\end{proof}

\bigskip
\noindent
{\bf Acknowledgements.}

\bigskip
The work has been partially supported by the National Natural Science Foundation of China (Grant No. 11101039 and
Grant No. 11171044),  and Hunan Provincial Natural Science Foundation of China (Grant No.11JJ2001). The authors are grateful to an anonymous referee for helpful remarks and comments.

\end{document}